\documentclass[a4paper]{article}

\usepackage{amsmath}
\usepackage{amssymb}
\usepackage{amsthm}
\usepackage{graphicx}
\usepackage{bbm}
\usepackage{url}
\usepackage{color}

\newtheorem{theorem}{Theorem}
\newtheorem{proposition}[theorem]{Proposition}
\newtheorem{lemma}[theorem]{Lemma}

\newtheorem{corollary}[theorem]{Corollary}

\numberwithin{equation}{section}
\numberwithin{theorem}{section}

\title{Asymptotics of some generalized Mathieu series}

\author{Stefan Gerhold\thanks{S.~Gerhold gratefully acknowledges financial support from the Austrian Science Fund (FWF) under grant P~30750 and from OeAD under grant MK 04/2018. We thank Michael Drmota
	for very helpful comments.}  \\
	TU Wien  \\
	\and 
	\v{Z}ivorad Tomovski  \\
	Saints Cyril and Methodius University of Skopje \\
	\\
	Dedicated to Prof.\ Tibor Pog{\'a}ny on the occasion
of his 65th birthday
	}

\date{\today}

\begin{document}



\maketitle

\begin{abstract}
  We establish asymptotic estimates of Mathieu-type series defined by sequences with
  power-logarithmic or factorial behavior.
  By taking the Mellin transform, the problem is mapped to the singular
  behavior of  certain Dirichlet series, which is then translated into asymptotics
  for the original series. In the case of power-logarithmic sequences, we obtain precise
  first order asymptotics. For factorial sequences, a natural boundary of the Mellin transform
  makes the problem more challenging, but a direct elementary estimate gives reasonably
  precise asymptotics.
\end{abstract}

\section{Introduction and main results}

Define, for $\mu\geq0,$ $r > 0$ and sequences $\mathbf{a} =(a_n)_{n\geq 0},$ $\mathbf{b} =(b_n)_{n\geq 0},$
\begin{equation}\label{eq:def S gen}
   S_{\mathbf{a},\mathbf{b},\mu}(r) := \sum_{n=0}^\infty \frac{a_n}
    {\big(b_n+r^2\big)^{\mu+1}}.
\end{equation}
The parametrization (i.e., $r^2$ and not~$r$, $\mu+1$ and not~$\mu$)
is along the lines of~\cite{SrMeTo18}.
Assumptions on the sequences $\mathbf{a}$ and $\mathbf{b}$ will be specified below.
The study of such series began with 19th century work of Mathieu on elasticity of solid bodies,
and has produced a considerable amount of literature, much of which focuses on
integral representations and inequalities. See, e.g., \cite{SrMeTo18,SrTo04,To10} for some recent
results and many references. 
As a special case of~\eqref{eq:def S gen}, define,
for $\alpha,\beta,r>0,$ $\mu\geq0$ with $\alpha-\beta(\mu+1)<-1$ and $\gamma,\delta \in \mathbb R,$
\begin{equation}\label{eq:def S}
  S_{\alpha,\beta,\gamma,\delta,\mu}(r) := \sum_{n=2}^\infty \frac{n^\alpha (\log n)^\gamma}
    {\big(n^\beta (\log n)^\delta+r^2\big)^{\mu+1}}.
\end{equation}
Note that the summation in~\eqref{eq:def S} starts at~$2$ to make the summand always well-defined.
The series~\eqref{eq:def S} is closely related to a paper by Paris~\cite{Pa13} (see also~\cite{Za09}),
but the presence of logarithmic factors is new.
Another special case of~\eqref{eq:def S gen} is the series
\begin{equation}\label{eq:def S!}
  S^!_{\alpha,\beta,\mu}(r) := \sum_{n=0}^\infty \frac{(n!)^\alpha}
    {\big((n!)^\beta +r^2\big)^{\mu+1}},
\end{equation}
defined for $\alpha,\mu\geq0$, $\beta,r>0$ with $\alpha-\beta(\mu+1)<0.$
We are not aware of any asymptotic estimates for~\eqref{eq:def S!} in the literature.
See~\cite{To10} for integral representations for some series of this kind.
The subject of the present paper is the asymptotic behavior of the Mathieu-type series~\eqref{eq:def S}
and~\eqref{eq:def S!} for $r\uparrow \infty$.
For the classical Mathieu series, the asymptotic expansion
\[
  \sum_{n=1}^\infty \frac{n}{(n^2+r^2)^2} \sim
    \sum_{k=0}^\infty (-1)^k \frac{B_{2k}}{2r^{2k+2}}, \quad r\uparrow\infty,
\]
was found be Elbert~\cite{El82}, whereas Pog\'{a}ny et al.~\cite{PoSrTo06} showed the expansion
\[
  \sum_{n=1}^\infty (-1)^{n-1}\frac{n}{(n^2+r^2)^2} \sim
  \sum_{k=1}^\infty  \frac{G_{2k}}{4r^{2k+2}}, \quad r\uparrow\infty,
\]
for its alternating counterpart; the~$B_n$ and $G_n$ are Bernoulli resp.\ Genocchi numbers.
We refer to~\cite{Pa13} for further references on asymptotics
of Mathieu-type series, to which we add \S19 and \S20 of~\cite{Fo60}.
To formulate our results on~\eqref{eq:def S}, for
\begin{equation}\label{eq:N}
  \delta (\alpha+1)/\beta -\gamma \notin\mathbb{N}=\{1,2,\dots\},
\end{equation}
we define the constant
\begin{multline*}
  C_{\alpha,\beta,\gamma,\delta,\mu}:= 
    \frac{(\tfrac12 \beta)^{\delta (\alpha+1)/\beta -\gamma-1}
      \Gamma\big(\frac{\delta}{\beta} (\alpha+1) -\gamma+1\big)}{2\Gamma(\mu+1)\Gamma\big({-\frac{\delta}{\beta} (\alpha+1)} +\gamma+1\big)} \\
    \times \Gamma\big({-\frac{\alpha+1}{\beta}}+\mu+1\big) \Gamma\big(\frac{\alpha+1}{\beta}\big).
\end{multline*}
If, on the other hand,
$m:=\delta (\alpha+1)/\beta -\gamma \in \mathbb N$ is a positive integer, then we define
\begin{equation}\label{def:C int}
  C_{\alpha,\beta,\gamma,\delta,\mu}:= \frac{\beta^{m-1}\Gamma\big({-\frac{\alpha+1}{\beta}}+\mu+1\big) \Gamma\big(\frac{\alpha+1}{\beta}\big)}{2^m\Gamma(\mu+1)}.
\end{equation}
\begin{theorem}\label{thm:main}
  Let $\alpha,\beta>0,$ $\mu\geq0,$ with $\alpha-\beta(\mu+1)<-1$, and $\gamma,\delta \in \mathbb R$. Then
  we have
  \begin{equation}\label{eq:main}
    S_{\alpha,\beta,\gamma,\delta,\mu}(r) \sim
      C_{\alpha,\beta,\gamma,\delta,\mu}\, r^{2 (\alpha+1)/\beta -2(\mu+1)} (\log r)^{-\delta (\alpha+1)/\beta +\gamma},
      \quad r\uparrow \infty.
  \end{equation}
\end{theorem}
Of course, the exponent of~$r$ is negative:
\[
  2 (\alpha+1)/\beta -2(\mu+1) = \frac{2}{\beta}\big(\alpha+1-\beta(\mu+1)\big) <0.
\]
Also, we note that for $\gamma=\delta=0$ (no logarithmic factors),
condition~\eqref{eq:N} is always satisfied,
and the asymptotic equivalence~\eqref{eq:main} agrees with a special
case of Theorem~3 in~\cite{Pa13}.
A bit more generally than Theorem~\ref{thm:main}, we have:
\begin{theorem}\label{thm:main2}
  Let the parameters $\alpha,\beta,\gamma,\delta,\mu$ be as in Theorem~\ref{thm:main}. 
  Let $\mathbf{a}$ and $\mathbf{b}$ be positive sequences that satisfy
  \[
    a_n \sim n^\alpha (\log n)^\gamma, \quad b_n \sim n^\beta (\log n)^\delta, \quad
      n \uparrow \infty.
  \]
  Then $S_{\mathbf{a},\mathbf{b},\mu}(r) $ has the asymptotic behavior stated in
  Theorem~\ref{thm:main}, i.e.
  \[
    S_{\mathbf{a},\mathbf{b},\mu}(r) \sim
      C_{\alpha,\beta,\gamma,\delta,\mu}\,
        r^{2 (\alpha+1)/\beta -2(\mu+1)} (\log r)^{-\delta (\alpha+1)/\beta +\gamma},
      \quad r\uparrow \infty.
  \]
\end{theorem}
This result includes sequences of the form $(\log n!)^\alpha$, see
Corollary~\ref{cor:log}. Also, it clearly implies that shifts such as $a_n=(n+a)^\alpha (\log(n+b))^\beta$
are not visible in the first order asymptotics. Theorems~\ref{thm:main}
and~\ref{thm:main2} are proved in Section~\ref{se:log}.
The series~\eqref{eq:def S!} is more difficult to analyze than~\eqref{eq:def S} by Mellin
transform (see Section~\ref{se:dirichlet}), but it turns out that it is asymptotically
dominated by only two summands. This yields the following result, which is proved
in Section~\ref{se:! proof}.
It uses an expansion for the functional inverse of the gamma function which is stated,
but not proved in~\cite{BoCo18}; see Section~\ref{se:! proof} for details.
We therefore state the theorem conditional on this expansion.
We write $\{ x\}$ for the fractional part of a real number~$x$.
\begin{theorem}\label{thm:!}
  Assume that the expansion of the inverse gamma function
  stated in equation~(70) of~\cite{BoCo18} is correct.
  Let $\alpha,\beta>0,$ $\mu\geq0,$ with $\alpha-\beta(\mu+1)<0$, and
  $0<d_1 < d_2 <1$.
  Then
  \begin{equation}\label{eq:! est}
    S^!_{\alpha,\beta,\mu}(r) = r^{-2(\mu+1-\alpha/\beta)}
      \exp\big({ -m(r)} \log \log r + O(\log \log \log r)\big)
  \end{equation}
  as $r\to\infty$ in the set
  \begin{equation}\label{eq:r set}
    \mathcal{R} := \big\{ r  > 0 : d_1 \leq \{\Gamma^{-1}(r^{2/\beta})\} \leq d_2\big\},
  \end{equation}
  where the function $m(\cdot)$ is defined by
  \[
    m(r) := \min\Big\{ \alpha\{\Gamma^{-1}(r^{2/\beta})\},
    \big(\beta(\mu+1)-\alpha\big)\big(1-\{\Gamma^{-1}(r^{2/\beta})\}\big) \Big\} > 0.
  \]
\end{theorem}
Thus, under the constraint~\eqref{eq:r set}, the series
$S^!_{\alpha,\beta,\mu}(r)$ decays like $r^{-2(\mu+1-\alpha/\beta)}$, accompanied by a
power of $\log r$, where the exponent of the latter depends on~$r$ and fluctuates in a finite
interval of negative numbers.
The expression inside the fractional part $\{\cdot\}$
growths roughly logarithmically:
\begin{equation}\label{eq:Gamma inv}
  \Gamma^{-1}(r^{2/\beta}) \sim \frac{2 \log r}{\beta \log \log r}, \quad r\uparrow\infty.
\end{equation}
Clearly, the proportion $\lim_{r\uparrow \infty} r^{-1} \mathrm{meas}(\mathcal{R} \cap [0,r])$ of
``good'' values of~$r$ can be made arbitrarily close to~$1$ by choosing~$d_1$
and $1-d_2$ sufficiently small.
Without the Diophantine assumption~\eqref{eq:r set}, a more complicated asymptotic expression
for $S^!_{\alpha,\beta,\mu}(r)$ is obtained by combining~\eqref{eq:le A}, \eqref{eq:A0},
and~\eqref{eq:A1 comp} below. From this expression it is easy to see
that, for any $\varepsilon>0$, we have
\begin{equation}\label{eq:eps as}
  r^{2\alpha/\beta-2(\mu+1) -\varepsilon} \ll  S^!_{\alpha,\beta,\mu}(r)\ll
  r^{2\alpha/\beta-2(\mu+1) +\varepsilon},
  \quad r\uparrow \infty,
\end{equation}
as well as logarithmic asymptotics:
\begin{equation}\label{eq:log as}
  \log S^!_{\alpha,\beta,\mu}(r) = -2(\mu+1-\alpha/\beta) \log r
  + O(\log \log r),\quad r\uparrow \infty.
\end{equation}
The following result contains an asymptotic upper bound; like~\eqref{eq:eps as}
and~\eqref{eq:log as}, it is valid without restricting~$r$ to~\eqref{eq:r set}:
\begin{theorem}\label{thm:! bd}
    Assume that the expansion of the inverse gamma function
  stated in equation~(70) of~\cite{BoCo18} is correct.
   Let $\alpha,\beta>0,$ $\mu\geq0,$ with $\alpha-\beta(\mu+1)<0.$
   Then
   \begin{equation}\label{eq:! bd}     
     S^!_{\alpha,\beta,\mu}(r) \leq r^{-2(\mu+1-\alpha/\beta)}
       \exp\big( o(\log \log r) \big), \quad
       r\to\infty.
   \end{equation}
\end{theorem}
Theorem~\ref{thm:! bd} is proved in Section~\ref{se:! proof}, too.
In Section~\ref{se:dirichlet}, we show the following unconditional bound,
which also holds for $\alpha=0$:
\begin{theorem}\label{thm:sp bd}
Let $\alpha,\mu\geq0,$ $\beta>0$  with $\alpha-\beta(\mu+1)<0.$ Then
\begin{equation*}
 S^!_{\alpha,\beta,\mu}(r) =O\Big( r^{-2(\mu+1-\alpha/\beta)}
       \frac{\log r}{\log \log r}\Big), 
     \quad r\uparrow \infty.
\end{equation*}
\end{theorem}
The difficulties concerning the factorial Mathieu-type series
stem from the fact that the Mellin transform of $S^!_{\alpha,\beta,\mu}(\cdot)$
has a natural boundary in the form of a vertical line,
whereas that of $S_{\alpha,\beta,\gamma,\delta,\mu}(\cdot)$ is more regular,
featuring an analytic continuation with a single branch cut. See Sections~\ref{se:log}
and~\ref{se:dirichlet} for details. We therefore prove Theorem~\ref{thm:!} by a direct estimate;
see Section~\ref{se:! proof}. It will be clear from the proof
that the error term in~\eqref{eq:! est} can be refined, if desired.
Also, $d_1$ and $1-d_2$ may depend on~$r$, as long as they tend to zero
sufficiently slowly.
%
%

\section{Power-logarithmic sequences}\label{se:log}

Since~\eqref{eq:def S} is a series with positive terms,
the discrete Laplace method seems to be a natural asymptotic tool; see~\cite{Pa11} for
a good introduction and further references. However, while the summands of~\eqref{eq:def S} do have
a peak around $n\approx r^{2/\beta}$, the local expansion of the summand does not fully
capture the asymptotics, and the central part of the sum yields an incorrect constant factor.
A similar phenomenon has been observed in \cite{DrSo95,FrGe15} for \emph{integrals} that are not amenable
to the Laplace method. As in~\cite{Pa13}, we instead use a Mellin transform approach.
Since the Mellin transform seems not to be explicitly available in our case, we invoke
results from~\cite{GrTh96} on the analytic continuation of a certain Dirichlet series.
Before beginning with the Mellin transform analysis, we show
that Theorem~\ref{thm:main2} follows from Theorem~\ref{thm:main}. This
is the content of the following lemma.
\begin{lemma}\label{le:simplify}
  Let $\mathbf a$ and $\mathbf b$ be as in Theorem~\ref{thm:main2}. Then
  \begin{equation*}
     S_{\mathbf{a},\mathbf{b},\mu}(r) = S_{\alpha,\beta,\gamma,\delta,\mu}(r) \big(1+o(1)\big)
       +O\big(r^{-2(\mu+1)}(\log r)^{2\alpha+1}\big), \quad r\uparrow \infty.
  \end{equation*}
\end{lemma}
\begin{proof}
  First consider the summation range $0\leq n\leq \lfloor\log r \rfloor$ for the series defining
  $ S_{\mathbf{a},\mathbf{b},\mu}(r) $.
  We have the estimate
  \begin{align*}
    b_n + r^2 &= O(n^\beta(\log n)^\gamma) + r^2 \\
    &= r^2(1 + O\big((\log r)^{2\beta}/r^2\big) \\
    &= r^2\big(1+o(1)\big), \quad r\uparrow \infty,
  \end{align*}
  and thus
  \[
    (b_n+r^2)^{-(\mu+1)} = r^{-2(\mu+1)}\big(1+o(1)\big), \quad
      0\leq n\leq \lfloor\log r \rfloor.
  \]
  We obtain
  \begin{align}
    \sum_{n=0}^{\lfloor\log r \rfloor} \frac{a_n}{\big(b_n+r^2\big)^{\mu+1}} &\lesssim
      \sum_{n=0}^{\lfloor\log r \rfloor} \frac{n^{2\alpha}}{\big(b_n+r^2\big)^{\mu+1}} \notag \\
    &\sim  r^{-2(\mu+1)} \sum_{n=0}^{\lfloor\log r \rfloor} n^{2\alpha} \notag \\
    &= O\big(r^{-2(\mu+1)}(\log r)^{2\alpha+1}\big). \label{eq:lower}
  \end{align}
  Now consider the range $\lfloor\log r \rfloor < n < \infty$, which yields the main contribution.
  As for the denominator, we have
  \begin{align}
    b_n + r^2 &= n^\beta (\log n)^\delta + o(n^\beta (\log n)^\delta) + r^2 \notag \\
    &=  (n^\beta (\log n)^\delta + r^2)
      \Big(1 + \frac{o(n^\beta (\log n)^\delta)}{n^\beta (\log n)^\delta + r^2}\Big) \notag \\
    &= (n^\beta (\log n)^\delta + r^2)\big(1+o(1) \big). \label{eq:bn}
  \end{align}
  Note that the first two $o(\cdot)$ are meant for $n\uparrow \infty$, but then the term
  $\frac{o(n^\beta (\log n)^\delta)}{n^\beta (\log n)^\delta + r^2}$ is also uniformly $o(1)$ as
  $r\uparrow\infty$, because $r\uparrow\infty$ implies $n\uparrow\infty$ in the range
  $\lfloor\log r \rfloor < n < \infty$. Similarly, we have
  \begin{equation}\label{eq:an}
    a_n = n^\alpha (\log n)^\gamma \big(1+o(1) \big), \quad r\uparrow \infty.
  \end{equation}
  Therefore,
  \begin{align}
    \sum_{n > \lfloor\log r \rfloor} \frac{a_n}{\big(b_n+r^2\big)^{\mu+1}}
      &\sim  \sum_{n > \lfloor\log r \rfloor} \frac{n^\alpha (\log n)^\gamma}
        {\big(n^\beta (\log n)^\delta + r^2\big)^{\mu+1}} \notag \\
    &= S_{\alpha,\beta,\gamma,\delta,\mu}(r) + O\big(r^{-2(\mu+1)}(\log r)^{2\alpha+1}\big).
      \label{eq:upper}
  \end{align}
  Here, the asymptotic equivalence follows from~\eqref{eq:bn} and~\eqref{eq:an}, and the equality
  follows from~\eqref{eq:lower}. The statement now follows by combining~\eqref{eq:lower} and~\eqref{eq:upper}.
\end{proof}
We now begin the proof of Theorem~\ref{thm:main}.
As in~\cite{GrTh96}, define the Dirichlet series
\begin{equation}\label{eq:dirichlet}
  \zeta_{\eta,\theta}(s) := \sum_{n=2}^\infty \frac{(\log n)^\eta}{(n (\log n)^\theta)^s},
  \quad \mathrm{Re}(s)>1,
\end{equation}
with real parameters $\eta,\theta$.
We will see below that the Mellin transform of~\eqref{eq:def S}
can be expressed using $\zeta_{\eta,\theta}(s)$.
The first two statements of the following
lemma are taken from~\cite{GrTh96}.
\begin{lemma}\label{le:dirichlet}
  The Dirichlet series $\zeta_{\eta,\theta}$ has an analytic continuation to the whole
  complex plane except $(-\infty,1]$. As $s\to 1$ in this domain, we have the asymptotics
  \[
     \zeta_{\eta,\theta}(s) \sim
     \begin{cases}
       \frac{(-1)^{m-1}}{(m-1)!} (s-1)^{m-1} \log \frac{1}{s-1} & \text{if }
         m=\theta-\eta \in \mathbb N, \\
       \Gamma(\eta-\theta+1)(s-1)^{\theta-\eta-1} & \text{otherwise.}
     \end{cases}
  \]
  The analytic continuation grows at most polynomially as
  $|\mathrm{Im}(s)|\uparrow\infty$ while $\mathrm{Re}(s)$ is bounded and positive.
\end{lemma}
\begin{proof}
  The statements about analytic continuation and asymptotics are proved
  in~\cite{GrTh96}. We revisit this proof in order to prove the polynomial estimate,
  which is needed later to apply Mellin inversion. By the Euler--Maclaurin summation formula,
  we have
  \begin{equation}\label{eq:euler}
    \zeta_{\eta,\theta}(s) = \int_2^\infty f(x) dx +\frac{f(2)}{2}
    -\sum_{k=1}^2 \frac{B_{2k}}{(2k)!}f^{(2k-1)}(2)
    - \int_2^\infty \frac{B_2(x-\lfloor x \rfloor)}{2} f^{(2)}(x)dx,
  \end{equation}
  where
  \[
    f(x) := \frac{(\log x)^\eta}{(x (\log x)^\theta)^s}
  \]
  and the $B$s are Bernoulli numbers resp.\ polynomials. As noted in~\cite{GrTh96},
  the last integral in~\eqref{eq:euler} is holomorphic in $\mathrm{Re}(s)>-1$,
  and applying the Euler--Maclaurin formula of arbitrary order yields the full
  analytic continuation, after analyzing the first integral in~\eqref{eq:euler}.
  To prove our lemma, it remains to estimate the growth of the terms in~\eqref{eq:euler}.
  The dominating factor of $f^{(2)}(x)$ satisfies
  \[
    \Big|\frac{\partial^2}{\partial x^2}x^{-s}\Big| = |s(s+1)| x^{-\mathrm{Re}(s)-2},
  \]
  from which it is very easy to see that the last integral in~\eqref{eq:euler} grows
  at most polynomially under the stated conditions on~$s$. In the first integral
  in~\eqref{eq:euler}, we substitute
  \begin{equation}\label{eq:subs}
    x=\exp\big(z/(s-1)\big)
  \end{equation}
  (as in~\cite{GrTh96}) and obtain
  \begin{align}
    \int_2^\infty f(x) dx &= (s-1)^{\theta s-\eta-1} \int_{(s-1)\log 2}^\infty
      z^{\eta-\theta s} e^{-z} dz \notag \\
      &= (s-1)^{\theta s-\eta-1}\bigg(\Gamma(\eta-\theta s+1)
        + \int_0^{(s-1)\log 2}z^{\eta-\theta s} e^{-z} dz \bigg). \label{eq:int}
  \end{align}
  From Stirling's formula, we have
  \[
    \Gamma(t) = O(e^{- \pi |\mathrm{Im}(t)|/2}|t|^{\mathrm{Re}(t)-1/2}),
    \quad |t|\uparrow \infty,
  \]
  uniformly w.r.t.~$\mathrm{Re}(t)$, as long as $\mathrm{Re}(t)$ stays bounded. Using this and
  \[
    |(s-1)^{-\theta s}| = \exp\big(-\theta\, \mathrm{Re}(s) \log |s-1|
    +\theta\, \mathrm{Im}(s) \arg(s-1) \big),
  \]
  we see that $|(s-1)^{\theta s-\eta-1}\Gamma(\eta-\theta s+1)|$ can be estimated
  by a polynomial in~$s$. Finally, we have
  \begin{multline*}
    \int_0^{(s-1)\log 2}z^{\eta-\theta s} e^{-z} dz \\
      = \big((s-1) \log 2\big)^{\eta-\theta s+1}
      \int_0^1 \big((s-1) u\log 2\big)^{\eta-\theta s}
      e^{(1-s)u \log 2}du,
  \end{multline*}
  from which it is immediate that the term
  \[
    (s-1)^{\theta s-\eta-1}\int_0^{(s-1)\log 2}z^{\eta-\theta s} e^{-z} dz
  \]
  in~\eqref{eq:int} admits a polynomial estimate.
\end{proof}

For any sufficiently regular function~$f$, we denote the Mellin transform by~$f^*$,
\[
  f^*(s) := \int_0^\infty f(r) r^{s-1} dr.
\]
We now compute the Mellin transform of the function $S_{\alpha,\beta,\gamma,\delta,\mu}(r)$,
writing  $a_n = n^\alpha (\log n)^\gamma$ and $b_n = n^\beta (\log n)^\delta.$
\begin{align}
  S^*_{\alpha,\beta,\gamma,\delta,\mu}(s) &=
    \int_{0}^\infty S_{\alpha,\beta,\gamma,\delta,\mu}(r) r^{s-1}dr \notag \\
    &= \sum_{n=2}^\infty a_n \int_{0}^\infty \frac{r^{s-1}}{(b_n+r^2)^{\mu+1}} dr\notag \\
    &= \frac12\sum_{n=2}^\infty a_n b_n^{s/2-(\mu+1)}
      \int_{0}^\infty\frac{u^{s/2-1}}{(1+u)^{\mu+1}}du\notag \\
   &=  \frac{D(s) \Gamma(\mu+1-s/2) \Gamma(s/2)}{2\Gamma(\mu+1)}, \label{eq:mellin}
\end{align}
where we substituted $u=r^2/b_n,$ and
\begin{align*}
  D(s) &:= \sum_{n=2}^\infty n^\alpha (\log n)^\gamma
  \big(n^\beta(\log n)^{\delta}\big)^{s/2-(\mu+1)}  \\
  &= \sum_{n=2}^\infty (\log n)^{ \delta s/2+\gamma-\delta(\mu+1)}
    n^{ \beta s/2+\alpha -\beta(\mu+1)}.  \\
\end{align*}
The Dirichlet series~$D$ can be expressed in terms of $\zeta_{\eta,\theta}$ from~\eqref{eq:dirichlet}:
\begin{equation}  \label{eq:D}
  D(s) = \zeta_{\eta,\theta}\big(1+\tfrac12 \beta(\hat{s}-s)\big)
  \Big|_{\eta = \gamma -\alpha \delta/\beta,\ \theta = \delta/\beta} 
\end{equation}
with
\begin{equation}\label{eq:hat s}
  \hat s := -2 (\alpha+1)/\beta +2(\mu+1) < 2\mu+2.
\end{equation}
Formula~\eqref{eq:mellin} is valid for $\mathrm{Re}(s) \in (0,\hat s)$.
The function $\Gamma(\mu+1-s/2)$ has poles at $2\mu+2,2\mu+4,\dots$, and
those of $\Gamma(s/2)$ are $0,-2,-4,\dots$ 
All those poles are outside the strip $\{s \in \mathbb{C} : \mathrm{Re}(s) \in (0,\hat s)\}.$
The singular expansion
of~\eqref{eq:mellin} at the dominating singularity~$\hat s$ can be translated,
via the Mellin inversion formula, into the asymptotic behavior of $S_{\alpha,\beta,\gamma,\delta,\mu}(r)$.
See~\cite{FlGoDu95} for a standard introduction to this method; in fact, our generalized
Mathieu series~\eqref{eq:def S gen} is a harmonic sum in the terminology of~\cite{FlGoDu95}.
By  Mellin inversion, we have
\begin{align}
  S_{\alpha,\beta,\gamma,\delta,\mu}(r) &=  \frac{1}{2\pi i} \int_{\kappa-i\infty}^{\kappa+i\infty}
    r^{-s} S^*_{\alpha,\beta,\gamma,\delta,\mu}(s) ds \notag \\
  &= \frac{1}{2\Gamma(\mu+1)}
  \frac{1}{2\pi i}\int_{\kappa-i\infty}^{\kappa+i\infty}
    r^{-s} D(s)\Gamma(\mu+1-s/2)\Gamma(s/2)ds, \label{eq:mint}
\end{align}
where $\kappa \in (0,\hat s)$. Note that integrability of $S^*_{\alpha,\beta,\gamma,\delta,\mu}(s)$
follows from the polynomial estimate in
Lemma~\ref{le:dirichlet} and Stirling's formula, as the latter implies
\begin{equation}\label{eq:stir im}
  \Gamma(t) = O\big(\exp(-(\tfrac12 \pi -\varepsilon) |\mathrm{Im}(t)|\big)
\end{equation}
for bounded $\mathrm{Re}(t)$.
Suppose first that
\[
  \theta-\eta=\delta (\alpha+1)/\beta -\gamma \notin \mathbb N.
\]
Then, from Lemma~\ref{le:dirichlet} and~\eqref{eq:D}, we have
\begin{align}
  D(s) &\sim \Gamma(\delta (\alpha+1)/\beta -\gamma+1)
   \big(\tfrac12 \beta(\hat s - s)\big)^{\delta (\alpha+1)/\beta -\gamma-1} \notag \\
  &= c_1 (\hat s - s)^{-c_2}, \quad s\to\hat s, \label{eq:D c1}
\end{align}
with
\begin{align}
  c_1 &:= \Gamma(\delta (\alpha+1)/\beta -\gamma+1)(\tfrac12 \beta)^{\delta (\alpha+1)/\beta -\gamma-1}, \notag \\
  c_2 &:= -\delta (\alpha+1)/\beta +\gamma+1. \label{eq:c2}
\end{align}
Combining~\eqref{eq:mellin} and~\eqref{eq:D c1} yields
\begin{equation}\label{eq:S^* asympt}
  S^*_{\alpha,\beta,\gamma,\delta,\mu}(s) \sim c_3  (\hat s - s)^{-c_2}, \quad s\to\hat s,
\end{equation}
where
\begin{equation}\label{eq:c3}
  c_3 :=   \frac{ c_1 \Gamma(\mu+1-\hat{s}/2) \Gamma(\hat{s}/2)}{2\Gamma(\mu+1)}.
\end{equation}
By a standard procedure, we can now extract
asymptotics of the Mathieu-type series $S_{\alpha,\beta,\gamma,\delta,\mu}(r)$ from~\eqref{eq:mint}.
The integration contour in~\eqref{eq:mint}
is pushed to the right, which is allowed by Lemma~\ref{le:dirichlet}. 
The real part of the new contour is
\[
  \kappa_r := \hat s + \frac{\log \log r}{\log r},
\]
where the singularity
at $s=\hat s$ is avoided by a small C-shaped notch. In~\eqref{eq:notch} below,
this notch is the integration contour. The contour is then transformed
to a Hankel contour~$\mathcal H$ by the substitution $s=\hat s-w/\log r$.
The contour~$\mathcal H$ starts at $-\infty$, circles the origin counterclockwise
and continues back to $-\infty$.
Using~\eqref{eq:S^* asympt}, we thus obtain
\begin{align}
  S_{\alpha,\beta,\gamma,\delta,\mu}(r) &= 
  \frac{1}{2\pi i}\int_{\kappa_r-i\infty}^{\kappa_r+i\infty}
    r^{-s} S^*_{\alpha,\beta,\gamma,\delta}(s)ds \notag \\
  &\sim \frac{c_3}{2\pi i} \int r^{-s} (\hat s - s)^{-c_2} ds \label{eq:notch} \\
  &\sim c_3 r^{-\hat s} (\log r)^{c_2-1}\frac{1}{2\pi i}\int_{\mathcal H} e^w w^{-c_2} dw \notag \\
  &= \frac{c_3}{\Gamma(c_2)}r^{-\hat s} (\log r)^{c_2-1}. \notag
\end{align}
See~\cite{FlGoDu95,FlOd90,FrGe15,GrTh96} for details of this asymptotic transfer. This completes the 
proof of~\eqref{eq:main} in the case $\delta (\alpha+1)/\beta -\gamma \notin \mathbb N$.
Recall the definitions of the constants $\hat{s}, c_2, c_3$
in~\eqref{eq:hat s}, \eqref{eq:c2}, and~\eqref{eq:c3}.

Now suppose that
\begin{equation}\label{eq:m}
  m:=\theta-\eta=\delta (\alpha+1)/\beta -\gamma \in \mathbb N.
\end{equation}
We need to show that~\eqref{eq:main} still holds, but with the constant
factor now given by~\eqref{def:C int}.
By Lemma~\ref{le:dirichlet} and~\eqref{eq:D}, we have
\begin{align}
  D(s) &\sim \frac{(-1)^{m-1}}{(m-1)!} (\tfrac12 \beta)^{m-1}(\hat s- s)^{m-1}
    \log \frac{1}{\hat s- s} \notag \\
  &= c_4 (\hat s- s)^{m-1} \log \frac{1}{\hat s- s}, \quad s\to\hat s, \label{eq:D as}
\end{align}
where
\[
  c_4 := \frac{(-1)^{m-1}}{(m-1)!} (\tfrac12 \beta)^{m-1}.
\]
Define
\begin{equation}\label{eq:c5}
  c_5 := \frac{c_4  \Gamma(\mu+1-\hat{s}/2) \Gamma(\hat{s}/2)}{2\Gamma(\mu+1)}.
\end{equation}
Then, using~\eqref{eq:mellin} and~\eqref{eq:D as},
\begin{equation*}
  S^*_{\alpha,\beta,\gamma,\delta,\mu}(s) \sim c_5  (\hat s - s)^{m-1}
    \log \frac{1}{\hat s- s}, \quad s\to\hat s.
\end{equation*}
We proceed similarly as above (see again~\cite{FlGoDu95,FlOd90,GrTh96}) and find
\begin{align}
  S_{\alpha,\beta,\gamma,\delta,\mu}(r) &\sim c_5 r^{-\hat s}(\log r)^{-m}
    \frac{1}{2\pi i} \int_{\mathcal H} e^w w^{m-1} \big(\log \frac{e^r}{w}\big) dw \notag\\
  &\sim c_5 r^{-\hat s}(\log r)^{-m}
    \frac{1}{2\pi i} \int_{\mathcal H} e^w w^{m-1} (-\log w) dw \notag \\
  &= c_5  \Big(\frac{1}{\Gamma}\Big)'(1-m) \times r^{-\hat s}(\log r)^{-m}. \label{eq:N res}
\end{align}
As for the second $\sim$, note that
\[
  \frac{1}{2\pi i} \int_{\mathcal H} e^w w^{m-1}  dw = \frac{1}{\Gamma(1-m)} = 0.
\]
From the well-known residues of~$\Gamma$ and~$\psi$ at the non-positive integers
(see, e.g., p.241 in~\cite{WhWa96}), we obtain
\[
  \Big(\frac{1}{\Gamma}\Big)'(1-m) = 
  -\Big(\frac{\psi}{\Gamma}\Big)(1-m)  =(-1)^{m-1} (m-1)!, \quad m\in\mathbb N.
\]
Formula~\eqref{eq:main} is established, and Theorem~\ref{thm:main}
is proved.
As for the constants in~\eqref{eq:N res}, recall the definitions in~\eqref{eq:hat s},
\eqref{eq:m}, and~\eqref{eq:c5}.
As mentioned above, Theorem~\ref{thm:main2} follows from Theorem~\ref{thm:main}
and Lemma~\ref{le:simplify}.

\section{Factorial sequences: the associated Dirichlet series}\label{se:dirichlet}

In the Mellin transform of~\eqref{eq:def S!}, the following Dirichlet series occurs:
\begin{equation}\label{eq:eta}
  \eta(s) := \sum_{n=0}^\infty (n!)^{-s}, \quad \mathrm{Re}(s) >0.
\end{equation}
As we will see in Lemma~\ref{le:dirichlet!}, this function does not have an analytic continuation beyond
the right half-plane.
It is well known that the presence
of a natural boundary is a severe obstacle when doing asymptotic transfers; see~\cite{FlFuGoPaPo06}
and the references cited there.
Therefore, our proof of Theorem~\ref{thm:!} in Section~\ref{se:! proof} will \emph{not} use Mellin
transform asymptotics. Still, some analytic properties of~\eqref{eq:eta}
seem to be interesting in their own right, and will be discussed in the present section.
We note that the arguments at the beginning of the proof of Lemma~\ref{le:dirichlet!} 
(analyticity, natural boundary) suffice to identify the \emph{location}
of the singularity of the Mellin transform
of $S^!_{\alpha,\beta,\mu}(\cdot)$ (see~\eqref{eq:inv! explicit} below), and
thus yield the logarithmic asymptotics in~\eqref{eq:log as} with the weaker error term $o(\log r)$.
Moreover, in this section we will prove Theorem~\ref{thm:sp bd};
see~\eqref{eq:sp bd} below.
\begin{lemma}\label{le:dirichlet!}
  The function~$\eta$ is analytic in the right half-plane, and the imaginary axis $i\mathbb R$
  is a natural boundary. At the origin, we have the asymptotics
  \begin{equation}\label{eq:F as}
    \eta(s) \sim \frac{1}{s\log(1/s)}, \quad s \downarrow 0, \ s \in\mathbb R.
  \end{equation}
\end{lemma}
\begin{proof}
   Analyticity follows from a standard result on Dirichlet series, see e.g.\ p.5 in~\cite{HaRi15}.
   As $n/\! \log n!=o(1)$, the lacunary series $\sum _{n=0}^\infty z^{\log n!}$
   has the unit circle as a natural boundary. We refer to the introduction of~\cite{CoHu09}
   for details. This implies that $i\mathbb R$ is the natural boundary of
   \[
     \eta(s) = \sum _{n=0}^\infty z^{\log n!}\Big|_{z=e^{-s}}.
   \]
   It remains to prove~\eqref{eq:F as}. We begin by showing that the Dirichlet series
   \begin{equation}\label{eq:dir log}
      \sum _{n=2}^\infty (\log n!)^{-s}, \quad \mathrm{Re}(s)>1,
   \end{equation}
   has an analytic continuation to $\mathrm{Re}(s)>0$, with branch cut
   $(0,1]$. The main idea is that replacing
   $\log n!$ by $n\log n$ leads to the series from Lemma~\ref{le:dirichlet},
   and the properties of~\eqref{eq:dir log} that we need are the same as those stated there.
   We just do not care about continuation further left than $\mathrm{Re}(s)>0$, because we do not require it.
   The continuation of~\eqref{eq:dir log} is based on writing
   \begin{equation}\label{eq:dir log2}
     \sum _{n=3}^\infty (\log n!)^{-s} =
       \sum_{n=3}^\infty \Big(  (\log n!)^{-s} -(n\log n-n)^{-s} \Big) 
       + \sum_{n=3}^\infty (n\log n-n)^{-s}.
   \end{equation} 
   By Stirling's formula, we have
   \begin{align*}
     (\log n!)^{-s} &= (n\log n-n)^{-s}\big(1+O(1/n)\big)^{-s} \\
     &=  (n\log n-n)^{-s}\big(1+O(1/n)\big),
   \end{align*}
   locally uniformly w.r.t.~$s$ in the right half-plane. From this it follows that
   \[
     \sum_{n=3}^\infty \Big(  (\log n!)^{-s} -(n\log n-n)^{-s} \Big) 
   \]
   defines an analytic function of~$s$ for $\mathrm{Re}(s)>0$. Moreover, the last series
   in~\eqref{eq:dir log2} has an analytic continuation to a slit plane. This is proved by the same argument
   as in Lemma~\ref{le:dirichlet}, using the Euler-Maclaurin formula
   and~\eqref{eq:subs}. Moreover, the polynomial estimate from that lemma
   easily extends to the continuation of~\eqref{eq:dir log} for $\mathrm{Re}(s)>0,$
   $s\notin(0,1]$.   
   After these preparations we can prove~\eqref{eq:F as} by Mellin transform asymptotics.
   We compute, recalling the definition of~$\zeta_{\eta,\theta}$ in~\eqref{eq:dirichlet}
   and its asymptotics from Lemma~\ref{le:dirichlet},
   \begin{align}
     \Big(s \sum_{n=2}^\infty (n!)^{-s}\Big)^*(t) &= \sum_{n=2}^\infty \int_0^\infty (n!)^{-s}s^t ds \notag \\
     &= \Gamma(t+1) \sum_{n=2}^\infty (\log n!)^{-t-1}  \label{eq:dir used} \\
     &\sim \Gamma(t+1) \sum_{n=2}^\infty (n \log n)^{-t-1} \notag \\
     &= \Gamma(t+1)\, \zeta_{0,1}(t+1)  \notag\\
     &\sim \log \frac1t, \quad t\to0. \notag
  \end{align}
  We have shown above that the Dirichlet series in~\eqref{eq:dir used} has an analytic continuation
  to $\mathrm{Re}(t)>-1$, $t\notin (-1,0]$, and so Lemma~2 in~\cite{GrTh96}
  is applicable (asymptotic transfer, with $a=0$, $b=1$ in the notation of~\cite{GrTh96}).
  We conclude
  \[
  s \sum_{n=2}^\infty (n!)^{-s} \sim \Big(\log \frac1s \Big)^{-1},  \quad s\downarrow0,\ s\in\mathbb R,
  \]
  and hence
  \[
    \eta(s) \sim \frac{1}{s \log(1/s)}, \quad s\downarrow0,\ s\in\mathbb R. \qedhere
  \]
\end{proof}
Analogously to~\eqref{eq:mellin}, we find the Mellin transform of~\eqref{eq:def S!}:
\begin{align}
  S^{!\, *}_{\alpha,\beta,\mu}(s) &= 
  \int_{0}^\infty S^!_{\alpha,\beta,\mu}(r) r^{s-1}dr \notag \\
  &=  \frac{ \Gamma(\mu+1-s/2) \Gamma(s/2)}{2\Gamma(\mu+1)}
    \sum_{n=0}^\infty (n!)^\alpha (n!)^{\beta(s/2-(\mu+1))} \notag \\
  &= \frac{ \Gamma(\mu+1-s/2) \Gamma(s/2)}{2\Gamma(\mu+1)}
    \eta\big(\tfrac12 \beta(\tilde{s}-s)\big), \label{eq:inv! explicit}
\end{align}
where
\begin{equation}\label{eq:s tilde}
  \tilde s :=  2(\mu+1 -\alpha/\beta)>0.
\end{equation}
By the Mellin inversion formula, we have
\begin{equation}\label{eq:inv!}
  S^!_{\alpha,\beta,\mu}(r) = \frac{1}{2\pi i} \int_{\sigma-i\infty}^{\sigma+i\infty}
    r^{-s} S^{!\,  *}_{\alpha,\beta,\mu}(s) ds,
    \quad 0<\sigma<\tilde s.
\end{equation}
Note that integrability of the Mellin transform $S^{!\,  *}_{\alpha,\beta,\mu}$ follows
from~\eqref{eq:stir im} and the obvious estimate
\begin{equation}\label{eq:abs}
  |\eta(s)| \leq \eta\big(\mathrm{Re}(s)\big), \quad \mathrm{Re}(s)>0.
\end{equation}
By~\eqref{eq:inv! explicit} and Lemma~\ref{le:dirichlet!}, the integrand
in~\eqref{eq:inv!} has a singularity at $s=\tilde s$, with singular expansion
\begin{equation}\label{eq:sing}
  \log\Big( r^{-s} S^{!\, *}_{\alpha,\beta,\mu}(s) \Big)
  =-s\log r + \log\frac{1}{\tilde{s}-s} - \log \log \frac{1}{\tilde{s}-s} + O(1).
\end{equation}
It is well known that this kind of singularity (polynomial growth
of the transform) is \emph{not} amenable to the saddle
point method, as regards precise asymptotics. Still, a \emph{saddle point bound} can be readily
found.
For an introduction to saddle point bounds and the saddle point method, we recommend
Chapter~VIII in~\cite{FlSe09}. Retaining only the first two
terms on the right-hand side of~\eqref{eq:sing} and taking the derivative w.r.t.~$s$ yields
the saddle point equation
\[
  \log r= \frac{1}{\tilde{s}-s}, 
\]
with solution
\begin{equation}\label{eq:def sigma}
  \sigma_r := \tilde{s} - \frac{1}{\log r}.
\end{equation}
We take this as real part of the integration path in~\eqref{eq:inv!} and obtain,
using~\eqref{eq:abs},
\begin{align}\label{eq:bd S}
  |S^!_{\alpha,\beta,\mu}(r)| &\leq
  r^{-\sigma_r}   \eta\big(\tfrac12 \beta(\tilde{s}-\sigma_r)\big)
   \frac{1}{2\pi } \int_{-\infty}^{\infty}
   \frac{ |\Gamma(\mu+1-s/2) \Gamma(s/2)|}{2\Gamma(\mu+1)}\Big|_{s=\sigma_r+iy}  dy \notag \\
   &= O \Big(r^{-\sigma_r}   \eta\big(\tfrac12 \beta(\tilde{s}-\sigma_r)\big) \Big).
\end{align}
The fact that the integral is $O(1)$ as $r\uparrow \infty$
follows from~\eqref{eq:stir im}. From~\eqref{eq:def sigma}, we have
\begin{equation}\label{eq:bd r}
  r^{-\sigma_r} = e r^{-\tilde{s}}.
\end{equation}
Lemma~\ref{le:dirichlet!} implies
\begin{equation} \label{eq:bd F}
  \eta\big(\tfrac12 \beta(\tilde{s}-\sigma_r)\big)
  =\eta\Big(\frac{\beta }{2\log r}\Big) 
  \sim \frac{2\log r}{\beta \log \log r},
\end{equation}
which results in the saddle point bound
\begin{equation}\label{eq:sp bd}
 S^!_{\alpha,\beta,\mu}(r) =O\Big( r^{-2(\mu+1-\alpha/\beta)}
       \frac{\log r}{\log \log r}\Big), 
     \quad r\uparrow \infty,
\end{equation}
which proves Theorem~\ref{thm:sp bd}.
Note that this bound is weaker than~\eqref{eq:! bd}, but does not require
the -- so far not proven -- expansion~\eqref{eq:G inv} of the inverse gamma function.
The saddle point bound~\eqref{eq:sp bd} also holds for $\alpha=0$, which is excluded in Theorems~\ref{thm:!}
and~\ref{thm:! bd}, because our proof of~\eqref{eq:right tail} below
requires $\alpha>0$.

\section{Factorial sequences: Proofs}\label{se:! proof}

This section contains the proofs of Theorems~\ref{thm:!} and~\ref{thm:! bd}.
Our estimates can be viewed as a somewhat degenerate instance of the Laplace method,
where the central part of the sum consists of just two summands.
We denote by~$A_n$ the summands of~\eqref{eq:def S!}:
\[
  S^!_{\alpha,\beta,\mu}(r) = \sum_{n=0}^\infty A_n,
  \qquad
  A_n := \frac{(n!)^\alpha}{\big((n!)^\beta +r^2\big)^{\mu+1}}.
\]
Define $n_0 = n_0(r)$ by $n_0(r):=\lfloor\Gamma^{-1}(r^{2/\beta})\rfloor-1$, i.e.,
\begin{equation}\label{eq:n0}
  (n_0!)^\beta \leq r^{2} < (n_0+1)!^\beta.
\end{equation}
We first show that $S^!_{\alpha,\beta,\mu}(r)$ is dominated by $A_{n_0}$ and
$A_{n_0+1}$.
For brevity, we omit writing the dependence of $A_n$ and $n_0$ on~$r$.
\begin{lemma}\label{le:A}
  Let $\alpha,\beta>0,$ $\mu\geq0,$ with $\alpha-\beta(\mu+1)<0$. Then
  \begin{equation}\label{eq:le A}
    S^!_{\alpha,\beta,\mu}(r) \sim A_{n_0} + A_{n_0+1}, \quad r\uparrow \infty.
  \end{equation}
\end{lemma}
\begin{proof}
  For $k\geq2$, we estimate, using~\eqref{eq:n0},
  \begin{align*}
    A_{n_0+k}/A_{n_0+1} &= \big((n_0+2)\dots(n_0+k)\big)^\alpha
    \bigg(
      \frac{(n_0+1)!^\beta + r^2}{(n_0+k)!^\beta + r^2}
    \bigg)^{\mu+1} \\
  &\leq \big((n_0+2)\dots(n_0+k)\big)^\alpha
     \bigg(
      \frac{2(n_0+1)!^\beta}{(n_0+k)!^\beta}
    \bigg)^{\mu+1} \\
  &= 2^{\mu+1} \big((n_0+2)\dots(n_0+k)\big)^{\alpha-\beta(\mu+1)}.
  \end{align*}
  Therefore,
  \begin{align*}
    A_{n_0+1}^{-1} \sum_{k=2}^\infty A_{n_0+k}
    &\leq 2^{\mu+1} \sum_{k=2}^\infty \big((n_0+2)\dots(n_0+k)\big)^{\alpha-\beta(\mu+1)} \\
    &\leq 2^{\mu+1} \sum_{k=2}^\infty n_0^{(k-1)(\alpha-\beta(\mu+1))} \\
    &\sim   2^{\mu+1} n_0^{\alpha-\beta(\mu+1)} = o(1).
  \end{align*}
  This shows that
  \[
    \sum_{k=2}^\infty A_{n_0+k} \ll A_{n_0+1}.
  \]
  For the initial segment $\sum_{k=1}^{n_0-1} A_{n_0-k}$
  of the series, we use the following estimate for $k\geq1$:
  \begin{align*}
    A_{n_0-k}/A_{n_0} &= \big(n_0(n_0-1)\dots(n_0-k+1)\big)^{-\alpha}
    \bigg(
      \frac{(n_0!)^\beta + r^2}{(n_0-k)!^\beta + r^2}
    \bigg)^{\mu+1} \\
   &\leq \big(n_0(n_0-1)\dots(n_0-k+1)\big)^{-\alpha}
    \bigg(
      \frac{2r^2}{ r^2}
    \bigg)^{\mu+1} \\
   &= 2^{\mu+1} \big(n_0(n_0-1)\dots(n_0-k+1)\big)^{-\alpha} =: 2^{\mu+1} B_k.
   \end{align*}
   Pick an integer $q$ with $q>1/\alpha$. Then
   \begin{equation}\label{eq:B sum}
     \sum_{k=1}^{n_0} B_k
     = \sum_{k=1}^{q}  B_k + \sum_{k=q+1}^{n_0}  B_k.
   \end{equation}
   Now $\sum_{k=1}^{q}  B_k$ has a fixed number of summands, all $o(1)$, and is
   thus $o(1)$ as $r\uparrow \infty$. In the second sum, we pull out the
   factor $n_0^{-\alpha}$, estimate $q$ of the remaining factors by $n_0-k+1$,
   and the other factors by~$1$:
   \begin{align*}
     \sum_{k=q+1}^{n_0}  B_k
     &\leq n_0^{-\alpha} \sum_{k=q+1}^{n_0} (n_0-k+1)^{-\alpha q} \\
     &\leq n_0^{-\alpha} \sum_{k=1}^{n_0} k^{-\alpha q} = O(n_0^{-\alpha}).
   \end{align*}
   The last equality follows from $q>1/\alpha$.
   We conclude that~\eqref{eq:B sum} is $o(1)$, and thus
   \begin{equation}\label{eq:right tail}
    \sum_{k=1}^{n_0-1} A_{n_0-k} \ll A_{n_0},
   \end{equation}
   which finishes the proof.
\end{proof}
We now evaluate $A_{n_0}$ and  $A_{n_0+1}$ asymptotically. We use the following notation,
partially in line with p.417f.\ of~\cite{BoCo18}, where asymptotic inversion
of the gamma function is discussed. We write $W(\cdot)$ for the Lambert $W$
function, which satisfies $W(z)\exp(W(z))=z$.
\begin{align}
  x &:= r^{2/\beta}, \quad v:= x/\sqrt{2\pi}, \notag \\
  g &:= \Gamma^{-1}(x), \notag \\
  n_0 &= \lfloor g \rfloor -1 = g - \{ g \}-1, \label{eq:n0 g}\\
  w &:= W\big((\log v)/e\big), \notag \\
  u_0 &:= (\log v)/w. \label{eq:u0 def}
\end{align}
It is easy to check, using the defining property of Lambert~$W$, that
\begin{equation}\label{eq:u0}
  u_0 \log u_0 - u_0 = \log v;
\end{equation}
in fact, this is equation~(63) in~\cite{BoCo18}.
\begin{proof}[Proof of Theorem~\ref{thm:!}]
By Stirling's formula and~\eqref{eq:n0 g}, we have
\begin{align}
  \log n_0! &= n_0 \log n_0 - n_0 + \tfrac12 \log n_0 + O(1) \notag \\
  &= (g - \{ g \}-1)\big(\log g + O(1/g)\big) - g  +\tfrac12 \log g + O(1) \notag \\
  &= g \log g - g -(\tfrac12 + \{g\}) \log g + O(1). \label{eq:log n0!}
\end{align}
As mentioned in Theorems~\ref{thm:!} and~\ref{thm:! bd}, we require the expansion
\begin{equation}\label{eq:G inv}
  \Gamma^{-1}(x) = u_0 + \frac12 + O\Big(\frac{1}{u_0 w}\Big)
\end{equation}
of the inverse gamma function; see equation~(70) in~\cite{BoCo18} (stated there, with
an additional term, but  without proof).
Note that first order asymptotics $\Gamma^{-1}(x) \sim u_0$,
i.e.~\eqref{eq:Gamma inv}, are very easy to
prove using the approach of~\cite{BoCo18}, just by carrying the $O(1/u)$ term neglected
after equation~(62) in~\cite{BoCo18} a few lines further.
{}From~\eqref{eq:log n0!} and~\eqref{eq:G inv}, we obtain
\begin{align*}
  \log n_0! &= u_0 \log u_0 - u_0 + \tfrac12 \log u_0 -(\tfrac12 + \{g\}) \log u_0 + O(1) \\
  &= u_0 \log u_0 - u_0 - \{g\} \log u_0 + O(1).
\end{align*}
Together with~\eqref{eq:u0}, this yields
\begin{align}
  n_0! &= v \exp\big({- \{g\}} \log u_0 + O(1)\big) \notag \\
  &= r^{2/\beta}\exp\big({- \{g\}} \log u_0 + O(1)\big). \label{eq:n0!}
\end{align}
Equation~\eqref{eq:n0!} is crucial for determining the asymptotics of
the right hand side of~\eqref{eq:le A}. Since
\[
   \exp\big({- \beta\{g\}} \log u_0 + O(1)\big) +1= e^{O(1)},
\]
we can use~\eqref{eq:n0!} to evaluate the summand $A_{n_0}$ as
\begin{align}
  A_{n_0} &= \frac{(n_0!)^\alpha}{\big((n_0!)^\beta +r^2\big)^{\mu+1}} \notag \\
  &= r^{2\alpha/\beta - 2(\mu+1)}
  \exp\big({- \alpha\{g\}} \log u_0 + O(1)\big) \notag\\
  &= r^{2\alpha/\beta - 2(\mu+1)}
  \exp\big({- \alpha\{g\}} \log \log r + O(\log \log \log r)\big). \label{eq:A0}
\end{align}
In the last line, we used the fact that
\begin{equation}\label{eq:W}
  W(z) \sim \log z, \quad z\uparrow\infty,
\end{equation}
see~\cite{CoGoHaJeKn96}.
The definition of~$n_0$ (see~\eqref{eq:n0}), \eqref{eq:G inv}, and~\eqref{eq:W}
imply
\[
  n_0 \sim \frac{2 \log r}{\beta \log \log r}, \quad r\uparrow \infty.
\]
As for the summand $A_{n_0+1}$, we thus have (with $\log^3 = \log \log \log$)
\begin{align}
  A_{n_0+1} &= \frac{(n_0+1)!^\alpha}{\big((n_0+1)!^\beta +r^2\big)^{\mu+1}}\notag \\
  &=  \frac{(\log r)^\alpha e^{O(\log^3 r)} (n_0!)^\alpha}{\big((\log r/\log \log r)^\beta
   e^{O(1)} (n_0!)^\beta +r^2\big)^{\mu+1}}\notag\\
 &= (\log r)^\alpha r^{2\alpha/\beta}\exp\big({- \alpha\{g\}} \log u_0 + O(\log^3 r)\big) \notag\\
 & \qquad \times \bigg(\Big(\frac{\log r}{\log \log r}\Big)^\beta r^2
   \exp\big({- \beta\{g\}} \log u_0 + O(1)\big) +r^2\bigg)^{-(\mu+1)} \notag\\
 &= r^{2\alpha/\beta - 2(\mu+1)} \exp\big(\alpha(1-\{g\}) \log \log r + O(\log^3 r)\big) \notag \\
 & \qquad \times \bigg(\Big(\frac{\log r}{\log \log r}\Big)^\beta 
   \exp\big({- \beta\{g\}} \log u_0 + O(1)\big) +1\bigg)^{-(\mu+1)}. \label{eq:A1 comp} 
\end{align}
This holds as $r\to\infty$, without any constraints on~$r$.
If $\{g\}\leq d_2 < 1$, as assumed in Theorem~\ref{thm:!},
then the term inside the big parentheses in~\eqref{eq:A1 comp}  goes to infinity;
note that $\log u_0 \sim \log \log r$ by~\eqref{eq:u0 def} and~\eqref{eq:W}.
We then have
\begin{equation}\label{eq:A1}
  A_{n_0+1} = r^{2\alpha/\beta - 2(\mu+1)}
    \exp\big((\alpha-\beta(\mu+1))(1-\{g\}) \log \log r + O(\log^3 r)\big).
\end{equation}
Define
\[
  \mathcal{R}_0 := \Big\{ r \in\mathcal{R} : -\alpha\{g\}  \geq
    (\alpha-\beta(\mu+1))(1-\{g \}) \Big\}
\]
and
\[
  \mathcal{R}_1 := \mathcal{R} \setminus \mathcal{R}_0.
\]
Then, by Lemma~\ref{le:A}, \eqref{eq:A0}, and~\eqref{eq:A1}, we obtain
\begin{align}
  S^!_{\alpha,\beta,\mu}(r)  &\sim A_{n_0}, \quad r\in  \mathcal{R}_0, \label{eq:S A0} \\
  S^!_{\alpha,\beta,\mu}(r)  &\sim A_{n_0+1}, \quad r\in  \mathcal{R}_1. \label{eq:S A1}
\end{align}
Theorem~\ref{thm:!} now follows from this,
\eqref{eq:A0}, and~\eqref{eq:A1}. Note that the assumption $0<d_1\leq \{ g\}\leq d_2<1$
of Theorem~\ref{thm:!}
ensures that the term $(\dots)\log \log r$ in~\eqref{eq:A0} and~\eqref{eq:A1} asymptotically
dominates the error term. Moreover, the asymptotic equivalence in~\eqref{eq:S A0}
and~\eqref{eq:S A1} can be replaced by an equality, because the error factor $1+o(1)$
is absorbed into the $O(\log^3 r)$ in the exponent.
\end{proof}
\begin{proof}[Proof of Theorem~\ref{thm:! bd}]
By~\eqref{eq:A0}, we have
\[
  A_{n_0} \leq r^{2\alpha/\beta - 2(\mu+1)} \exp\big( O(\log \log \log r)\big),
\]
and so, by Lemma~\ref{le:A}, it suffices to estimate~$A_{n_0+1}$. 
Fix an arbitrary $\varepsilon>0$. Recall the notation
introduced around~\eqref{eq:n0 g}. If $r$ is such that $\alpha(1-\{g\}) \leq \varepsilon$,
then we simply estimate the term in big parentheses in~\eqref{eq:A1 comp} by~$1$,
and obtain
\[
  A_{n_0+1} \leq r^{2\alpha/\beta - 2(\mu+1)} \exp\big(
  \varepsilon \log \log r + O(\log \log \log r)\big).
\]
If, on the other hand, $\{g\} < 1- \varepsilon/\alpha$, then~\eqref{eq:A1} holds,
which implies
\[
  A_{n_0+1} \leq r^{2\alpha/\beta - 2(\mu+1)} \exp\big( O(\log \log \log r)\big),
\]
because the quantity in front of $\log \log r$ in~\eqref{eq:A1} is negative.
We have thus shown that, for any $\varepsilon>0$,
\begin{equation}\label{eq:S bound}
   S^!_{\alpha,\beta,\mu}(r) \leq r^{-2(\mu+1-\alpha/\beta)}
       \exp\big( \varepsilon \log \log r + O(\log \log \log r) \big)
\end{equation}
From this, Theorem~\ref{thm:! bd} easily follows. Indeed, were it not true,
then there would be $\varepsilon'>0$ and a sequence $r_n\uparrow\infty$ such that
\[
  \log\big(r_n^{2(\mu+1-\alpha/\beta)}S^!_{\alpha,\beta,\mu}(r_n)\big)
  \geq  2\varepsilon' \log \log r_n,
\]
contradicting~\eqref{eq:S bound}.
\end{proof}

\section{Power sequences: Full expansion in a special case}

In~\cite{SrMeTo18}, an integral representation of the generalized Mathieu series
\[
  S_{\mu}(r) := \sum_{n=1}^\infty \frac{2n}{(n^2+r^2)^{\mu+1}}, \quad \mu>\tfrac32,r>0,
\]
was derived. In our notation, this series is 
\[
  S_{\mu}(r) = 2 S_{1,2,0,0,\mu}(r) + \frac{2}{(1+r^2)^{\mu+1}}.
\]
We use said integral representation and Watson's lemma
to find a full expansion of $S_{\mu}(r)$ as $r\to\infty$.
This expansion is not new (see Theorem~1 in~\cite{Pa13}), and so we do not give full details.
Still, our approach provides an independent check for (a special case of) Theorem~1 in~\cite{Pa13},
and it might be useful for other Mathieu-type series admitting a representation
as a Laplace transform.
The integral representation in Theorem~4 of~\cite{SrMeTo18} is
\begin{equation}\label{eq:S int}
  S_{\mu}(r) = c_\mu \int_0^\infty e^{-rt} t^{\mu+1/2}g_\mu(t) dt,
\end{equation}
where
\[
  c_\mu := \frac{\sqrt{\pi}}{2^{\mu-1/2}\Gamma(\mu+1)},
\]
and $g_\mu$ is the Schl\"omilch series
\[
  g_\mu(t) := \sum_{n=1}^\infty n^{1/2-\mu} J_{\mu+1/2}(nt).
\]
For $\mathrm{Re}(s) > \tfrac32-\mu$, the Mellin transform of $g_\mu$ is
\begin{align*}
  g_\mu^*(s) &= \sum_{n=1}^\infty n^{1/2-\mu-s}2^{s-1}
    \frac{\Gamma(\mu/2+1/4+s/2)}{\Gamma(\mu/2+5/4-s/2)} \\
  &= \frac{2^{s-1}\zeta(s+\mu-1/2)\Gamma(\mu/2+1/4+s/2)}{\Gamma(\mu/2+5/4-s/2)}.
\end{align*}
The factor $\zeta(s+\mu-1/2)$ has a pole at $\check{s} := \tfrac32-\mu$, and
$\Gamma(\mu/2+1/4+s/2)$ has poles at $s_k:= -2k-\mu-\tfrac12$, $k\in\mathbb{N}_0$.
By using Mellin inversion and collecting residues, we find
that the expansion of  $g_\mu(t)$ as $t\downarrow0$ is
\begin{align*}
  g_\mu(t) &\sim \frac{2^{\check{s}-1}\Gamma(\mu/2+1/4+\check{s}/2)}
    {\Gamma(\mu/2+5/4-\check{s}/2)} t^{-\check{s}}
  + \sum_{k=0}^\infty \frac{(-1)^k2^{s_k} \zeta(s_k+\mu-1/2)}
    {k!\, \Gamma(\mu/2+5/4-\check{s}/2)} t^{-s_k} \\
  &= \frac{2^{1/2-\mu}}{\Gamma(\mu+1/2)}t^{\mu-3/2}
  +\sum_{k=0}^\infty \frac{(-1)^k 2^{-2k-\mu-1/2} \zeta(-2k-1)}
    {k!\, \Gamma(k+\mu+3/2)} t^{2k+\mu+1/2}.
\end{align*}
No we multiply this expansion by $t^{\mu+1/2}$ and use Watson's lemma (\cite{Ol74}, p.71)
in~\eqref{eq:S int}. In the notation of \cite{Ol74}, p.71, the parameters~$\mu$ and~$\lambda$
are $\tfrac12$ and our~$\mu$, respectively.
Simplifying the resulting expansion using Legendre's duplication formula,
\[
  \Gamma(2k+2\mu+2) =\pi^{-1/2} 2^{2k+2\mu+1} \Gamma(k+\mu+1)\Gamma(k+\mu+3/2),
\]
yields the expansion
\begin{equation}\label{eq:expans}
  S_\mu(r) \sim \frac{1}{\mu} r^{-2\mu} + \sum_{k=0}^\infty
    \frac{2(-1)^k \zeta(-2k-1)\Gamma(k+\mu+1)}{\Gamma(\mu+1)k!}
    r^{-2k-2\mu-2}
\end{equation}
as $r\to\infty$. Recall that the values of the zeta function at negative odd integers
can be represented by Bernoulli numbers:
\[
  \zeta(-2k-1) = -\frac{B_{2k+2}}{2k+2}, \quad k\in\mathbb{N}_0.
\]
The expansion~\eqref{eq:expans} indeed agrees with  Theorem~1 in~\cite{Pa13}, and the first term
agrees with our Theorem~\ref{thm:main} (with $\alpha=1$, $\beta=2$, $\gamma=\delta=0$).
The divergent series in~\eqref{eq:expans} looks very similar to
formula~(3.2) in~\cite{Sr88}, but there the argument of $\zeta(\cdot)$ in the summation is
eventually positive instead of negative.

Finally, we give an amusing non-rigorous derivation of the asymptotic series on
the right-hand side of~\eqref{eq:expans},
by using the binomial theorem, the ``formula'' $\zeta(-2k-1)=\sum_{n=1}^\infty n^{2k+1},$
and interchanging summation:
\begin{align*}
  S_\mu(r) &= 2 r^{-2(\mu+1)} \sum_{n=1}^\infty n \Big(1+\frac{n^2}{r^2}\Big)^{-(\mu+1)} \\
  &=  2 r^{-2(\mu+1)} \sum_{n=1}^\infty n
    \sum_{k=0}^\infty(-1)^k\binom{k+\mu}{k} \Big(\frac{n}{r}\Big)^{2k} \\
  ``\!&=\!" \ \ 2 r^{-2(\mu+1)}  \sum_{k=0}^\infty (-1)^k\binom{k+\mu}{k}
    r^{-2k} \zeta(-2k-1) \\
  &= \sum_{k=0}^\infty
    \frac{2(-1)^k \zeta(-2k-1)\Gamma(k+\mu+1)}{\Gamma(\mu+1)k!}
    r^{-2k-2\mu-2}.
\end{align*}
Note that the dominating term of order $r^{-2\mu}$ is not found by this heuristic.

\section{Application and further comments}

We now apply Theorem~\ref{thm:main2} (on power-logarithmic sequences)
to an example taken from~\cite{To10}. There,
integral representations for some Mathieu-type series were deduced, and we  state
asymptotics for one of them.
\begin{corollary}\label{cor:log}
  Let $\alpha,\beta>0,$ $\mu\geq0,$ with $\alpha-\beta(\mu+1)<-1$. Then
  \[
    \sum_{n=2}^\infty \frac{(\log n!)^\alpha}{\big((\log n!)^\beta +r^2\big)^{\mu+1}}
      \sim C\, r^{2(\alpha+1)/\beta-2(\mu+1)}/\log r,\quad r\uparrow \infty,
  \]
  with
  \[
    C = \frac{\Gamma\big({-\frac{\alpha+1}{\beta}}+\mu+1 \big)
      \Gamma\big( \frac{\alpha+1}{\beta} \big)}{2\Gamma(\mu+1)}.
  \]
\end{corollary}
\begin{proof}
  By Stirling's formula, we have $(\log n!)^\alpha \sim (n \log n)^\alpha$.
  The statement thus follows from Theorem~\ref{thm:main2}, with $\gamma=\alpha,$
  $\delta=\beta,$ and
  $m=\delta (\alpha+1)/\beta -\gamma=1 \in \mathbb N.$
\end{proof}
A natural generalization of our main results on power-logarithmic sequences
(Theorems~\ref{thm:main} and~\ref{thm:main2})
would be to replace $\log$ by
an arbitrary slowly varying function: $a_n= n^\alpha \ell_1(n)$, $b_n= n^\beta \ell_2(n)$.
Then the Dirichlet series~\eqref{eq:D} becomes
\begin{align*}
  D(s) &= \sum_{n=2}^\infty a_n b_n^{s/2-(\mu+1)} \\
  &=  \sum_{n=2}^\infty n^{\beta s/2+\alpha-\beta(\mu+1)}
    \ell_1(n) \ell_2(n)^{s/2-(\mu+1)}. 
\end{align*}
The dominating singularity is still~$\hat s$ defined in~\eqref{eq:hat s}, as follows
from Proposition~1.3.6 in~\cite{BiGoTe87}, but it
seems not easy to determine the singular behavior of~$D$ at~$\hat s$
for generic $\ell_1,\ell_2$. Still, for specific examples such as $(\log \log n)^{\gamma}$
or $\exp(\sqrt{\log n})$, this should be doable. Note that our second step,
i.e.\ the asymptotic transfer from the Mellin transform to the original function,
works for slowly varying functions under mild conditions; see~\cite{FlOd90}.

Finally, we note that introducing a geometrically decaying factor $x^n$
to the series~\eqref{eq:def S gen} leads to a Mathieu-type \emph{power series.}
According to the following proposition, its asymptotics can be found in an elementary way,
for rather general sequences $\mathbf a,\mathbf b$.
 We refer to~\cite{ToPo11} for integral representations and further references
on certain Mathieu-type power series.

\begin{proposition}\label{prop:ps}
  Let $x\in\mathbb C$ with $|x|<1$, $a_n\in \mathbb C$, $b_n\geq 0,$ and $\mu\geq0$.
  If  $\sum_{n=0}^\infty a_n x^n$ is absolutely convergent and $b_n\uparrow \infty$,
  then
  \[
    \sum_{n=0}^\infty \frac{a_n}{\big(b_n+r^2\big)^{\mu+1}}x^n
      = r^{-2(\mu+1)} \sum_{n=0}^\infty a_n x^n + o(r^{-2(\mu+1)}), \quad r\uparrow \infty.
  \]
\end{proposition}
\begin{proof}
  We have
  \begin{align*}
    \Big|\sum_{n:\, b_n>r} \frac{a_n}{\big(b_n+r^2\big)^{\mu+1}}x^n\Big|
      &\leq \sum_{n:\, b_n>r} \frac{|a_n|}{\big(b_n+r^2\big)^{\mu+1}}|x|^n \\
      &\leq \sum_{n:\, b_n>r} \frac{|a_n|}{r^{2(\mu+1)}}|x|^n.
  \end{align*}
  As $\sum_{n:\, b_n>r}|a_n||x|^n$ tends to zero, this is $o(r^{-2(\mu+1)})$.
  For the dominating part of the series, we find
  \begin{align*}
    \sum_{n:\, b_n \leq r} \frac{a_n}{\big(b_n+r^2\big)^{\mu+1}}x^n
      &= r^{-2(\mu+1)} \sum_{n:\, b_n \leq r} \frac{a_n}{\big(b_n/r^2+1\big)^{\mu+1}}x^n\\
    &= r^{-2(\mu+1)} \Big( \sum_{n:\, b_n \leq r} a_n x^n +O(1/r) \Big) \\
    &= r^{-2(\mu+1)} \Big( \sum_{n=0}^\infty a_n x^n +o(1) \Big).
  \end{align*}
  In the last equality, we used that $\sum_{n:\, b_n> r} a_n x^n=o(1)$, because $b_n\uparrow \infty$.
\end{proof}
In Proposition~\ref{prop:ps}, we assumed $|x|<1$. Our main results
(Theorems~\ref{thm:main}--\ref{thm:!})
are concerned with the case $x=1$, for some special sequences $\mathbf a,\mathbf b$.
An alternating factor $(-1)^n$, on the other hand, induces cancellations that are difficult
to handle, and usually requires the availability of an explicit Mellin transform,
as in~\cite{Pa13}.

\bibliographystyle{siam}
\bibliography{literature}

\end{document}